\documentclass{amsart}

\numberwithin{equation}{section}

\usepackage{hyperref}

\usepackage{graphicx,color,graphics}
\usepackage{latexsym} 
\usepackage{tikz}
\newtheorem{theorem}{Theorem}[section]

\theoremstyle{definition}
\newtheorem{definition}[theorem]{Definition}

\theoremstyle{remark}

\newtheorem{lemma}[theorem]{Lemma} 
\newtheorem{proposition}[theorem]{Proposition}

\def\Sat{\operatorname{Sat}}

\def\mys{\sigma}
\def\digr{{\mathcal G}}
\def\crit{\digr^c}

\def\R{\mathbb{R}}

\def\Digr{{\mathcal G}}

\if{

}\fi

\begin{document}

\title[An application of the max-plus spectral theory]{An application of the max-plus spectral theory to an ultradiscrete analogue of the Lax pair}

\author{Serge\u{\i} Sergeev}
\address{University of Birmingham, School of Mathematics,
Birmingham, Edgbaston B15 2TT}


\email{sergiej@gmail.com}
\thanks{The work is supported by EPSRC Grant RRAH15735 and RFBR-CRNF grant 11-01-93106.
It was initiated when the author was with the Max-Plus Team at INRIA and CMAP \'Ecole Polytechnique, France.}


\begin{abstract}
We study the ultradiscrete analogue of Lax pair 
proposed by Willox et al.\cite{Wil+}. This ``pair'' is a max-plus linear
system comprising four equations. Our starting point is to treat this system
as a combination of two max-plus eigenproblems, with two additional constraints.
Though infinite-dimensional, these two eigenproblems can be treated by means of the ``standard''
max-plus spectral theory. In particular, any solution to the system can be described as a max-linear combination
of fundamental eigenvectors associated with each soliton. We then describe the operation of undressing
using pairs of fundamental eigenvectors. We also study the solvability
of the complete system of four equations as proposed by Willox et al.~\cite{Wil+}.
\end{abstract}

\maketitle

\section{Introduction}
\subsection{Motivations and purposes}
We consider the system of four equations
\begin{equation}
\label{e:willox}
\begin{split}
&\max(\Phi_{l+1}^{(t)}-k,\ \Phi_{l-1}^{(t)})=\Phi_l^{(t)}+\max(U_{l-1}^{(t)}-1,\ -U_l^{(t)}),\\
&\max(\Phi_{l+1}^{(t+1)}-k,\ \Phi_{l-1}^{(t+1)})=\Phi_l^{(t+1)}+\max(U_l^{(t)}-1,-U_{l-1}^{(t)}),\\
&\max(\Phi_l^{(t)}+k-\omega,\ \Phi_{l-1}^{(t+1)}+U_l^{(t)}+k-1)=\Phi_{l+1}^{(t+1)},\\
&\max(\Phi_{l+1}^{(t+1)},\ \Phi_{l+1}^{(t)}+U_l^{(t)}-1)=\Phi_l^{(t)}.
\end{split}
\end{equation}
which appeared in the work of Willox et al.~\cite{Wil+}. Here we assume that the potential $U^{(t)}$ is known,
and that $U^{(t)}$ and the solutions $\Phi^{(t)},\Phi^{(t+1)}$ satisfy the conditions (A$U$) and (A$\Phi$) written below, see
Subsection~\ref{ss:sptheory}.

System~\eqref{e:willox} plays the role of the Lax pair for the ultradiscrete KdV equation
\begin{equation}
\label{e:cell-aut}
U_l^{(t+1)}=\min(1-U_l^{(t)},\sum_{k=-\infty}^{l-1} U_k^{(t)}-U_k^{(t+1)}),
\end{equation}
which describes the dynamics of Box\& Ball system of Takahashi and Satsuma~\cite{TS-90}. 
Willox et al.~\cite{Wil+} show how solving~\eqref{e:willox} helps to 
calculate the phase-shifts of solitons after interaction in the case of the real initial
$U$ and, more generally, to solve equation~\eqref{e:cell-aut} at all times.

Very briefly, the relation of~\eqref{e:cell-aut} to the classical discrete and continuous KdV equations
is as follows. It was shown by Tokihiro et al.~\cite{Tok+} that equation~\eqref{e:cell-aut} can be obtained as
ultradiscrete limit (or Maslov dequantization) of the discrete KdV equation
\begin{equation}
\label{e:discr-kdv}
\frac{1}{u_{l+1}^{(t+1)}}-\frac{1}{u_l^{(t)}}=\delta(u_{l+1}^{(t)}-u_l^{(t+1)}),\quad \delta>0
\end{equation}
written by Tsujimoto and Hirota~\cite{TH-98}. This equation turns into the famous Lotka-Volterra 
equation by taking the continuous limit ($\delta\to 0$)~\cite{Tok+}, and the Lotka-Volterra equation
is also known as an integrable discretization of the classical KdV equation. See~\cite{Tok+} and~\cite{Wil+} for 
more explanation.

The intention of this paper is to build a max-plus linear theory of~\eqref{e:willox}. To our point of view, such
theory is lacking in~\cite{Wil+}, where it is claimed that system~\eqref{e:willox} is always solvable, but without
going into the details of the proof. As we will see, the theory of system~\eqref{e:willox} is nontrivial and to the author's
knowledge this kind of problems never appeared in the max-plus literature and could be of its own interest. Namely,
we have two infinite max-plus eigenproblems represented by the first two equations of~\eqref{e:willox} (where the eigenvalue
is necessarily $0$), and two connections between
them represented by the last two equations. Thus we are led to study two related (but different) eigenproblems at the same time,
taking into account some additional constraints.

In this paper we do not address the solvability statement of~\cite{Wil+} in full strength, and rather concentrate on
developing the spectral theory associated with the first two equations of~\eqref{e:willox}. These are two closely related 
infinite max-plus eigenproblems of a special kind. The theory of such problems was developed by Akian, Gaubert and Walsh~\cite{AGW-05},
and it could be applied here. However, we notice that assuming conditions (A$U$) and (A$\Phi$) on $U$ and $\Phi$, the problem 
can be reduced to the more usual finite max-plus spectral theory as described in the monographs~\cite{BCOQ,But:10,HOW:05}. Namely
with each soliton of $U^{(t)}$ we can associate a pair of fundamental eigenvectors, and any solution of the first and of the second
equation of~\eqref{e:willox} appears as their max-plus linear combination, see Proposition~\ref{p:extension} and 
Theorem~\ref{t:fundam}.  Thus we describe the set of all solutions to the first two equations of~\eqref{e:willox}
with natural asymptotic behaviour. 

Next we consider the procedure of undressing the initial potential $U$ by means of a pair
of fundamental eigenvectors. It follows that in a natural special case when the ``interior'' of a soliton in $U$
consists of $1$'s, this soliton disappears after undressing, and the rest of the potential gets shifted by one position 
towards the soliton. Note that in the undressing procedure of~\cite{Wil+}, it is demonstrated that the other solitons may change 
their form. This effect does not happen in our case, since we use the fundamental eigenvectors.

Finally we treat the complete system~\eqref{e:willox}. In the case when $U$ has no massive solitons ($U^{(t)}_i+U_{i+1}^{(t)}<1$
for all $i$) or when $U$ has just one massive soliton, we confirm that~\eqref{e:willox} is solvable by showing
that any pair of fundamental eigenvectors is a solution. In the case of several massive solitons we show that to the contrary,
no pair of fundamental eigenvectors is a solution, so that a combination 
of these fundamental eigenvectors
satisfying~\eqref{e:willox} has to be guessed.

\subsection{Max-plus spectral theory}

Algebra max-plus is developed over the real numbers $\R$ completed by the least element $-\infty$,
with arithmetical operations $a\otimes b:=a+b$ (``multiplication'') and 
$a\oplus b:=\max(a,b)$ (``addition''). The new ``zero'' is $-\infty$ and the new ``unity'' is $0$. This arithmetics is
extended to matrices and vectors in the usual way so that
\begin{equation*}
(A\otimes B)_{ik}=\bigoplus_j A_{ij}\otimes B_{jk},\quad (A\oplus B)_{ij}=A_{ij}\oplus B_{ij},
\end{equation*}
for matrices $A$ and $B$ of appropriate sizes. We will be interested only in the max-plus spectral problem
\begin{equation}
\label{mp-spectral}
A\otimes\Phi=\lambda\otimes\Phi,
\end{equation}
that is, trying to find for a matrix $A\in(\R\cup\{-\infty\})^{n\times n}$ a parameter 
$\lambda\in\R\cup\{-\infty\}$ such that there exists a vector $\Phi$ satisfying~\eqref{mp-spectral} with
not all components equal to $-\infty$.

Max-plus spectral theory uses the following graph-theoretical concepts:\\
{\bf 1. Associated graph} $\digr(A)=(N,E)$ with set of nodes $N=\{1,\ldots,n\}$ and
set of edges $E=\{(i,j): A_{ij}\neq -\infty\}$ weighted by $w(i,j)=A_{ij}$. The concept of weight is
extended to paths $P=(i_0\to i_1\to\ldots\to i_k)$, 
defining the weight of $P$ by 
\begin{equation*}
w(P):=A_{i_0i_1}\cdot\ldots\cdot A_{i_{k-1}i_k}
\end{equation*}
Closed paths $P$ having $i_0=i_k$ are called {\bf cycles}.\\
{\bf 2. Critical graph} $\crit(A)$ comprising all nodes and edges that belong to the
cycles $(i_1,\ldots,i_k)$, on which the maximum in
\begin{equation}
\label{e:mcm}
\lambda(A)=\max\limits_{1\leq k\leq n}\max\limits_{1\leq i_1,\ldots,i_k\leq n} \frac{A_{i_1i_2}+\ldots+A_{i_ki_1}}{k}
\end{equation}
is attained. Such cycles are called {\bf critical}, and so are all nodes and edges of the critical graph.
Being made from cycles, 
the critical graph is completely reducible, i.e., it consists of several isolated strongly connected components.\\
{\bf 3. Saturation graph} $\Sat(\Phi)$ consisting of all nodes and edges satisfying $a_{ij}+\Phi_j=\lambda+\Phi_i$, that is,
attaining maximum on the l.h.s. of~\eqref{mp-spectral}.  

The following theorem explains some properties of the
saturation graph and its relation to the critical graph. It is well-known but we give
a short proof for the reader's convenience.

\begin{theorem}
\label{l:satprops} Let $A\in(\R\cup\{-\infty\})^{n\times n}$.
Suppose that $\Phi$ satisfies $A\otimes\Phi=\Phi$ and has all components finite (i.e.,
not $-\infty$). Then 
\begin{itemize}
\item[1.] Each node has an outgoing
edge in $\Sat(\Phi)$, 
\item[2.] Each cycle in $\digr$ has total weight not exceeding $0$, 
\item[3.] The cycles of $\Sat(\Phi)$ are precisely the cycles of $\crit(A)$
\end{itemize}
\end{theorem}
\begin{proof}
1.: If $\Phi$ is an eigenvector then for each $i$ there exists $j$ such that $A_{ij}+\Phi_j=\Phi_i$.

2. and 3.: Let $(i_1,\ldots, i_k)$ be a cycle in $\digr(A)$. Then we have
\begin{equation*}
A_{i_1i_2}+\Phi_{i_2}\leq \Phi_{i_1},\ldots,A_{i_ki_1}+\Phi_{i_1}\leq\Phi_{i_k}.
\end{equation*}
Combining these inequalities and cancelling $\Phi$ we get $A_{i_1i_2}+\ldots +A_{i_ki_1}\leq 0$,
which shows 2. Note that $A_{i_1i_2}+\ldots +A_{i_ki_1}<0$ is equivalent to having 
$A_{i_li_{l+1}}+\Phi_{i_{l+1}}<\Phi_{i_l}$ for some $l$, which implies 3.
\end{proof}

Note that Theorem~\ref{l:satprops} generalizes to the case when the matrix $A$ is infinite-dimensional but
each row has a finite number of real entries. This is the case that we will have to work with when analyzing~\eqref{e:willox}.

For $A\in(\R\cup\{-\infty\})^{n\times n}$, 
a formal analogue of $(I-A)^{-1}$ can be defined as
\begin{equation}
\label{kls-0}
A^*=I\oplus A\oplus A^2\oplus\ldots,
\end{equation}
where $I$ is the {\bf max-plus identity matrix}, having $0$ on the diagonal and $-\infty$ otherwise.
$A^*$ converges and can be truncated to $I\oplus A\oplus\ldots A^{n-1}$, if and only if
$\lambda(A)\leq 0$. Note that entries of $A^*$, denoted by $A^*_{ij}$, have a 
{\bf principal path interpretation}: for $i\neq j$ this is the greatest weight of a path connecting
$i$ to $j$.

When $\lambda(A)=0$ it can be shown that any column $A^*_{\cdot i}$ of $A^*$, whose index $i$ belongs to the critical graph
(i.e., {\bf critical column of $A^*$}), is an eigenvector of $A$. Such columns are called the 
{\bf fundamental eigenvectors}. The eigenspace of $A$ can be described more precisely as
follows.

\begin{theorem}
\label{t:max-spectrum}
Let $A\in(\R\cup\{-\infty\})^{n\times n}$.
If $\lambda(A)=0$ and $\Phi$ satisfies $A\otimes\Phi=\Phi$, then
there exist $\alpha_i\in\R\cup\{-\infty\}$ such that
\begin{equation*}
\Phi=\bigoplus_{i\in S} \alpha_i\otimes A_{\cdot i}^*,
\end{equation*}
where $S\in\{1,\ldots,n\}$ is any index set containing precisely one index from
each strongly connected component of $\crit(A)$.
\end{theorem}

That is, each eigenvector of $A$ is a {\bf max-linear combination} of the fundamental
eigenvectors.

Theorem~\ref{t:max-spectrum} can be found in several monographs on max-plus algebra~\cite{BCOQ, But:10, HOW:05}. 
The max-plus spectral theory (both finite- and infinite-dimensional) has applications ranging from railway scheduling~\cite{HOW:05} to
Frenkel-Kontorova model in solid state physics~\cite{CG-86},\cite{WCF-05}, and the crop rotation problem 
in the agriculture~\cite{Bac-03}.

\section{Simplified Lax pair}

\subsection{Solitons and critical graphs}

For most of this paper we will consider a simplified version of the Lax system for udKdV
consisting of the first two equations of~\eqref{e:willox}, which we rewrite as 

\begin{equation}
\label{e:lin-syst}
\begin{split}
\max[\Phi_{i+1}^{(t)}+\gamma_i-k,\Phi_{i-1}^{(t)}+\gamma_i] &=\Phi_i^{(t)},\\
\max[\Phi_{i+1}^{(t+1)}+\delta_i-k,\Phi_{i-1}^{(t+1)}+\delta_i] &=\Phi_i^{(t+1)},\\
\text{where} \gamma_i=\min(U_i^{(t)},1-U_{i-1}^{(t)}),\quad \delta_i &=\min(U_{i-1}^{(t)},1-U_i^{(t)}).
\end{split}
\end{equation}

Further we will fix $t$ and denote $u_i:=U_i^{(t)}$, $\Phi^{(1)}:=\Phi^{(t)}$ and $\Phi^{(2)}:=\Phi^{(t+1)}$.

We distinguish between two cases:\\
(C1) When $v_{\sup}=\sup_i (u_i+u_{i+1})\leq 1$.\\ 
(C2) When $v_{\sup}=\sup_i (u_i+u_{i+1})\geq 1$. 

Note that the borderline case $\sup_i u_i+u_{i+1}=1$ can be regarded in both ways, not leading to any contradiction.

We observe that~\eqref{e:lin-syst} is a combination of two max-plus
eigenproblems $A(\gamma)\otimes\Phi^{(1)}=\Phi^{(1)}$ (first equation) and 
$A(\delta)\otimes\Phi^{(2)}=\Phi^{(2)}$ (second equation), where the coefficients
of $A(\gamma)$ and $A(\delta)$ can be written as follows.

\begin{lemma}
\label{coeffs-s1}
In the case (C1),
\begin{equation}
\label{e:coeffs-s1}
\begin{split}
& A(\gamma)_{i+1,i}=u_{i+1},\quad A(\gamma)_{i,i+1}=u_i-k,\\
& A(\delta)_{i+1,i}=u_i,\quad A(\delta)_{i,i+1}=u_{i-1}-k.
\end{split}
\end{equation}
\end{lemma}

\begin{lemma}
\label{coeffs-s2}
In the case of (C2),
\begin{equation}
\label{e:coeffs-s2}
\begin{split}
&A(\gamma)_{i+1,i}=
\begin{cases}
u_{i+1}, &\text{if $u_i+u_{i+1}<1$},\\
1-u_i, &\text{if $u_i+u_{i+1}\geq 1$}
\end{cases},\\
&A(\gamma)_{i,i+1}=
\begin{cases}
u_i-k, &\text{if $u_i+u_{i-1}<1$},\\
1-u_{i-1}-k, &\text{if $u_i+u_{i-1}\geq 1$}
\end{cases},\\
&A(\delta)_{i+1,i}=
\begin{cases}
u_i, &\text{if $u_i+u_{i+1}<1$},\\
1-u_{i+1}, &\text{if $u_i+u_{i+1}\geq 1$}
\end{cases},\\ 
&A(\delta)_{i,i+1}=
\begin{cases}
u_{i-1}-k, &\text{if $u_i+u_{i-1}<1$},\\
1-u_i-k, &\text{if $u_i+u_{i-1}\geq 1$}
\end{cases}
\end{split}
\end{equation}
\end{lemma}

The proofs are straightforward. We proceed with the following crucial definition.

\begin{figure}
\centering

\begin{tikzpicture}[shorten >=1pt,->]
  \tikzstyle{vertex}=[circle,fill=black!15,minimum size=17pt,inner sep=1pt]
\foreach \name/\x in {1/3, 2/5, 3/7, 4/9}
    \node[vertex] (\name) at (\x,0) {$\name$};
\draw[xshift=1.5cm] node{$\Digr(\gamma):$};

\draw (1) .. controls +(50:1cm) and +(130:1cm) .. (2);
\draw (2) .. controls +(50:1cm) and +(130:1cm) .. (3);
\draw (3) .. controls +(50:1cm) and +(130:1cm) .. (4);

\draw[xshift=4cm,yshift=0.85cm] node{$u_{l-1}-k$};
\draw[xshift=6cm,yshift=0.85cm] node{$u_l-k$};
\draw[xshift=8cm,yshift=0.85cm] node{$u_{l+1}-k$};

\draw (4) .. controls +(-130:1cm) and +(-50:1cm) .. (3);
\draw (3) .. controls +(-130:1cm) and +(-50:1cm) .. (2);
\draw (2) .. controls +(-130:1cm) and +(-50:1cm) .. (1);

\draw[xshift=4cm,yshift=-0.85cm] node{$u_l$};
\draw[xshift=6cm,yshift=-0.85cm] node{$u_{l+1}$};
\draw[xshift=8cm,yshift=-0.85cm] node{$u_{l+2}$};
\end{tikzpicture}

\begin{tikzpicture}[shorten >=1pt,->]
  \tikzstyle{vertex}=[circle,fill=black!15,minimum size=17pt,inner sep=1pt]
\foreach \name/\x in { 1/3, 2/5, 3/7, 4/9}
    \node[vertex] (\name) at (\x,0) {$\name$};
\draw[xshift=1.5cm] node{$\Digr(\delta):$};

\draw (1) .. controls +(50:1cm) and +(130:1cm) .. (2);
\draw (2) .. controls +(50:1cm) and +(130:1cm) .. (3);
\draw (3) .. controls +(50:1cm) and +(130:1cm) .. (4);

\draw[xshift=4cm,yshift=0.85cm] node{$u_{l-2}-k$};
\draw[xshift=6cm,yshift=0.85cm] node{$u_{l-1} -k$};
\draw[xshift=8cm,yshift=0.85cm] node{$u_{l}-k$};

\draw (4) .. controls +(-130:1cm) and +(-50:1cm) .. (3);
\draw (3) .. controls +(-130:1cm) and +(-50:1cm) .. (2);
\draw (2) .. controls +(-130:1cm) and +(-50:1cm) .. (1);

\draw[xshift=4cm,yshift=-0.85cm] node{$u_{l-1}$};
\draw[xshift=6cm,yshift=-0.85cm] node{$u_{l}$};
\draw[xshift=8cm,yshift=-0.85cm] node{$u_{l+1}$};

\end{tikzpicture}

\begin{tikzpicture}[shorten >=1pt,->]

\draw[xshift=1.5cm] node{$U$:};

\draw[xshift=3cm,yshift=0cm] node{$u_{l-1}$};
\draw[xshift=5cm,yshift=0cm] node{$u_{l}$};
\draw[xshift=7cm,yshift=0cm] node{$u_{l+1}$};
\draw[xshift=9cm,yshift=0cm] node{$u_{l+2}$};

\end{tikzpicture}
\caption{\label{f:C1} Case (C1): a fragment of the associated digraphs}
\end{figure}
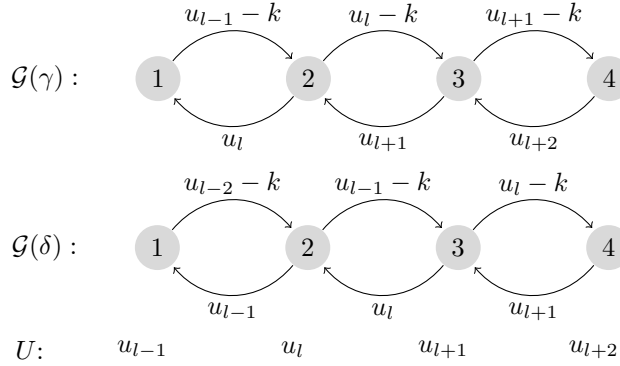

\begin{definition}[Solitons]
\label{def:sol}
\begin{enumerate}
\item[1.] In the case of (C1), soliton is a sequence of indices
$(l,l+1,\ldots,l+s)$ such that $u_l+u_{l+1}=\ldots=u_{l+s-1}+u_{l+s}=v_{\sup}$, while
$u_{l-1}+u_l< v_{\sup}$ and $u_{l+s}+u_{l+s+1}< v_{\sup}$. 
\item[2.] In the case of (C2), soliton is a sequence of indices $(l,l+1,\ldots,l+s)$ such that $u_l+u_{l+1}\geq 1,\ldots,
u_{l+s-1}+u_{l+s}\geq 1$, while $u_{l-1}+u_l< 1$ and $u_{l+s}+u_{l+s+1}<1$.
\end{enumerate}
\end{definition}

If in the equation of cellular automaton~\eqref{e:cell-aut}
we assume that $U_i^{(t)}=U_i^{(t+1)}=0$ for all $i<-N$ where $N$ is sufficiently
large (see condition $(AU)$ below), then its dynamics can be 
computed explicitly. In the case (C1) the whole vector $U^{(t)}$ gets shifted
by one position to the right. In the case (C2) the behaviour is more complex. 
Like in the classical theory of KdV, the solitons (as defined above) 
move with different speed depending on their mass (not defined here).
After interaction they emerge again with a phase-shift, as described by Willox et al.~\cite{Wil+}.
So it can be argued that Definition~\ref{def:sol} has a ``physical sense'' only in the
case (C2). However, as we show below, the theory of eigenproblems~\eqref{e:lin-syst} is similar in both cases.

Consider, with Lemmas~\ref{coeffs-s1} and~\ref{coeffs-s2} in mind, the associated weighted digraphs
$\Digr(\gamma)$ and $\Digr(\delta)$ of matrices $A(\gamma)$ and $A(\delta)$.
We are going to study the critical cycles, i.e., the two-cycles with the greatest total weight,
and the critical graph, consisting of all nodes and edges on the critical cycles. We relate
the strongly connected components of critical graphs to solitons, and we give a formula for 
the greatest total weight when the solitons exist.

The case (C1) is displayed on Figure~\ref{f:C1}. Clearly, solitons correspond to the strongly connected components of the
critical graph (if it is non-empty), consisting of the two-cycles with the greatest total weight
$\max_i (u_i+u_{i+1}-k)$.

In the case of (C2), we give only fragments of these digraphs corresponding to the tail (i.e., the left end)
and the head (i.e., the right end) of any soliton. The reader may assume $k=1$, which will follow from 
Proposition~\ref{p:sat-prop}, under some assumptions on $U$ and $\Phi$. See Figures~\ref{f:tail} and~\ref{f:head}.

\begin{figure}

\begin{tikzpicture}[shorten >=1pt,->]
  \tikzstyle{vertex}=[circle,fill=black!15,minimum size=17pt,inner sep=1pt]
\tikzstyle{vertex2}=[circle,fill=black!40,minimum size=17pt,inner sep=1pt]
\foreach \name/\x in {1/3, 4/9}
    \node[vertex] (\name) at (\x,0) {$\name$};
\foreach \name/\x in {2/5, 3/7}
    \node[vertex2] (\name) at (\x,0) {$\name$};

\draw[xshift=1.5cm] node{$\Digr(\gamma):$};

\draw (1) .. controls +(50:1cm) and +(130:1cm) .. (2);
\draw (2) .. controls +(50:1cm) and +(130:1cm) .. (3);
\draw (2) .. controls +(50:1cm) and +(130:1cm) .. (3);
\draw (3) .. controls +(50:1cm) and +(130:1cm) .. (4);

\draw[xshift=4cm,yshift=0.85cm] node{$u_{l-1}-k$};
\draw[xshift=6cm,yshift=0.85cm] node{$u_l-k$};
\draw[xshift=8cm,yshift=0.85cm] node{$1-u_l-k$};


\draw (4) .. controls +(-130:1cm) and +(-50:1cm) .. (3);
\draw (3) .. controls +(-130:1cm) and +(-50:1cm) .. (2);
\draw (3) .. controls +(-130:1cm) and +(-50:1cm) .. (2);
\draw (2) .. controls +(-130:1cm) and +(-50:1cm) .. (1);

\draw[xshift=4cm,yshift=-0.85cm] node{$u_l$};
\draw[xshift=6cm,yshift=-0.85cm] node{$1-u_l$};
\draw[xshift=8cm,yshift=-0.85cm] node{$1-u_{l+1}$};
\end{tikzpicture}

\begin{tikzpicture}[shorten >=1pt,->]
  \tikzstyle{vertex}=[circle,fill=black!15,minimum size=17pt,inner sep=1pt]
\foreach \name/\x in {1/3, 2/5, 3/7, 4/9}
    \node[vertex] (\name) at (\x,0) {$\name$};
\draw[xshift=1.5cm] node{$\Digr(\delta):$};

\draw (1) .. controls +(50:1cm) and +(130:1cm) .. (2);
\draw (2) .. controls +(50:1cm) and +(130:1cm) .. (3);
\draw (3) .. controls +(50:1cm) and +(130:1cm) .. (4);

\draw[xshift=4cm,yshift=0.85cm] node{$u_{l-2}-k$};
\draw[xshift=6cm,yshift=0.85cm] node{$u_{l-1}-k$};
\draw[xshift=8cm,yshift=0.85cm] node{$1-u_{l+1}-k$};

\draw (4) .. controls +(-130:1cm) and +(-50:1cm) .. (3);
\draw (3) .. controls +(-130:1cm) and +(-50:1cm) .. (2);
\draw (2) .. controls +(-130:1cm) and +(-50:1cm) .. (1);

\draw[xshift=4cm,yshift=-0.85cm] node{$u_{l-1}$};
\draw[xshift=6cm,yshift=-0.85cm] node{$1-u_{l+1}$};
\draw[xshift=8cm,yshift=-0.85cm] node{$1-u_{l+2}$};

\end{tikzpicture}

\begin{tikzpicture}[shorten >=1pt,->]

\draw[xshift=1.5cm] node{$U$:};

\draw[xshift=3cm,yshift=0cm] node{$u_{l-1}$};
\draw[xshift=5cm,yshift=0cm] node{$u_l$};
\draw[xshift=7cm,yshift=0cm] node{$u_{l+1}$};
\draw[xshift=9cm,yshift=0cm] node{$u_{l+2}$};
\end{tikzpicture}

\caption{\label{f:tail} Case (C2): tail of a soliton. Cycle (2,3) of $\Digr(\gamma)$ is critical.}
\end{figure}
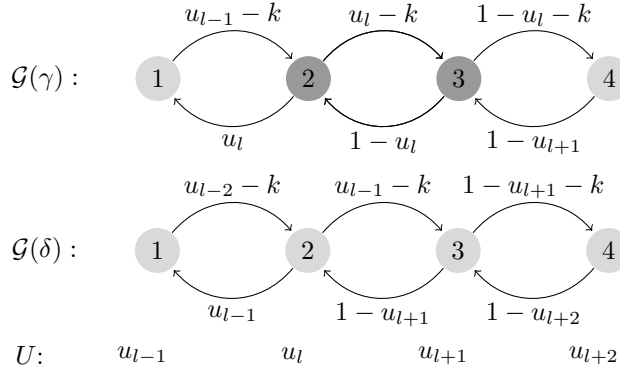

\begin{figure}

\begin{tikzpicture}[shorten >=1pt,->]
  \tikzstyle{vertex}=[circle,fill=black!15,minimum size=17pt,inner sep=1pt]

\foreach \name/\x in {1/3, 2/5.4, 3/7.8,  4/10.2}
    \node[vertex] (\name) at (\x,0) {$\name$};

\draw[xshift=1.5cm] node{$\Digr(\gamma):$};

\draw (1) .. controls +(50:1cm) and +(130:1cm) .. (2);
\draw (2) .. controls +(50:1cm) and +(130:1cm) .. (3);
\draw (3) .. controls +(50:1cm) and +(130:1cm) .. (4);

\draw[xshift=4.2cm,yshift=0.85cm] node{$1-u_{l+s-3}-k$};
\draw[xshift=6.6cm,yshift=0.85cm] node{$1-u_{l+s-2}-k$};
\draw[xshift=9cm,yshift=0.85cm] node{$1-u_{l+s-1}-k$};

\draw (4) .. controls +(-130:1cm) and +(-50:1cm) .. (3);
\draw (3) .. controls +(-130:1cm) and +(-50:1cm) .. (2);
\draw (2) .. controls +(-130:1cm) and +(-50:1cm) .. (1);

\draw[xshift=4.2cm,yshift=-0.85cm] node{$1-u_{l+s-2}$};
\draw[xshift=6.6cm,yshift=-0.85cm] node{$1-u_{l+s-1}$};
\draw[xshift=9cm,yshift=-0.85cm] node{$u_{l+s+1}$};
\end{tikzpicture}

\begin{tikzpicture}[shorten >=1pt,->]
  \tikzstyle{vertex}=[circle,fill=black!15,minimum size=17pt,inner sep=1pt]
\tikzstyle{vertex2}=[circle,fill=black!40,minimum size=17pt,inner sep=1pt]
\foreach \name/\x in {1/3, 2/5.4}
    \node[vertex] (\name) at (\x,0) {$\name$};
\foreach \name/\x in {3/7.8, 4/10.2}
    \node[vertex2] (\name) at (\x,0) {$\name$};

\draw[xshift=1.5cm] node{$\Digr(\delta):$};

\draw (1) .. controls +(50:1cm) and +(130:1cm) .. (2);
\draw (2) .. controls +(50:1cm) and +(130:1cm) .. (3);
\draw (3) .. controls +(50:1cm) and +(130:1cm) .. (4);
\draw (3) .. controls +(50:1cm) and +(130:1cm) .. (4);

\draw[xshift=4.2cm,yshift=0.85cm] node{$1-u_{l+s-2}-k$};
\draw[xshift=6.6cm,yshift=0.85cm] node{$1-u_{l+s-1} -k$};
\draw[xshift=9cm,yshift=0.85cm] node{$1-u_{l+s}-k$};

\draw (4) .. controls +(-130:1cm) and +(-50:1cm) .. (3);
\draw (4) .. controls +(-130:1cm) and +(-50:1cm) .. (3);
\draw (3) .. controls +(-130:1cm) and +(-50:1cm) .. (2);
\draw (2) .. controls +(-130:1cm) and +(-50:1cm) .. (1);

\draw[xshift=4.2cm,yshift=-0.85cm] node{$1-u_{l+s-1}$};
\draw[xshift=6.6cm,yshift=-0.85cm] node{$1-u_{l+s}$};
\draw[xshift=9cm,yshift=-0.85cm] node{$u_{l+s}$};

\end{tikzpicture}

\begin{tikzpicture}[shorten >=1pt,->]

\draw[xshift=1.5cm] node{$U$:};

\draw[xshift=3cm,yshift=0cm] node{$u_{l+s-2}$};
\draw[xshift=5.4cm,yshift=0cm] node{$u_{l+s-1}$};
\draw[xshift=7.8cm,yshift=0cm] node{$u_{l+s}$};
\draw[xshift=10.2cm,yshift=0cm] node{$u_{l+s+1}$};

\end{tikzpicture}

\caption{\label{f:head} Case (C2): head of a soliton. Cycle (3,4) of $\Digr(\delta)$
is critical}

\end{figure}
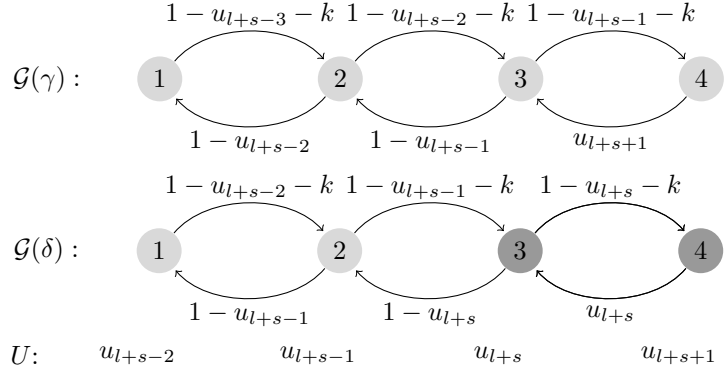

All simple cycles of these graphs have length two. If we are in the case (C2), then
the critical cycles of $\Digr(\gamma)$, i.e., the ones with the greatest sum of the weight of edges
equal to $1-k$,  
are in the tails of all massive solitons, between the nodes marked as 2 and 3. Likewise for $\Digr(\delta)$ the critical cycles (of the
same total weight $1-k$)
are in the heads of all massive solitons, between the nodes marked as 3 and 4. Other cycles in the
soliton can be also critical, if and only if $u_i+u_{i+1}=1$ for the corresponding $i$. We obtain that the greatest total weight of two-cycles is $1-k$. What we observed so far can be summarized as below.

\begin{theorem}
\label{t:first}
If the solitons exist, then the greatest total weight of two-cycles is
\begin{equation}
\label{kminmax}
\min(\max_i(u_{i-1}+u_i),1)-k.
\end{equation}
In this case the critical graphs of both $A(\gamma)$ and $A(\delta)$ are non-empty, and their strongly connected
components are in one-to-one correspondence with the solitons. 
\end{theorem}

\subsection{Reduction to the max-plus spectral theory}

\label{ss:sptheory}

Further we will assume the following requirements on the potential $U$
and on the solution $\Phi$.

(A$U$) There exists $N$ such that $u_i=0$ for all $i\geq N$ and $i\leq -N$.

(A$\Phi$) There exist arbitrarily large $N'$ and $N''$ such that $\Phi_{N'}=\Phi_{N'+1}$ 
and $\Phi_{-N''-1}=\Phi_{-N''}-k$.

It will be clear that $(A\Phi)$ is equivalent to the {\bf bound state} condition of~\cite{Wil+}:
that $\Phi_i$ tends to $-\infty$ when $i\to -\infty$, and that $\Phi_i$ is constant
for $i\geq N'$ for some $N'$.

In what follows we treat $A(\gamma)$ and $A(\delta)$ at the same time, denoting them
by $A$. The associated digraph will be denoted by $\digr$. We will need the 
following immediate observation (following, for instance, from Lemmas~\ref{coeffs-s1} and~\ref{coeffs-s2}).

\begin{lemma}
\label{l:tails}
We have $A_{i,i+1}=-k$ and $A_{i+1,i}=0$ for $i\geq N$ and $i\leq -N-1$. 
All weights of cycles $(i,i+1)$ equal $-k$.  These
cycles have weight $0$ (all of them) if and only if $k=0$.
\end{lemma}

With assumption (A$U$) we observe that $\digr$ always has cycles attaining the greatest total weight,
and that this weight is given by~\eqref{kminmax}.
Indeed, using Assumption (A$U$) we see that $u_{i-1}+u_i=0$ for all
$i\notin[-N+1,N]$. In the case when there is $i$ such that $u_{i-1}+u_i>0$,
this guarantees existence of solitons and leads to~\eqref{kminmax}. Otherwise the claim is trivial with
the greatest weight equal to $-k$ (see Lemma~\ref{l:tails}).

The saturation digraph of $\Phi$ can be introduced as in the introduction:
\begin{equation}
\label{satphi-def}
(i,j)\in\Sat(\Phi)\Leftrightarrow A_{ij}+\Phi_j=\Phi_i. 
\end{equation}

We now proceed with a proof (elementary but tedious) that with assumptions (A$U$) and (A$\Phi$), the 
solutions $\Phi$ are essentially the eigenvectors of the submatrix extracted from the interval $[-N-1,N+1]$.
In Proposition~\ref{p:sat-prop} we will show that $\Sat(\Phi)$ has an outgoing edge from all nodes in
$[-N-1, N+1]$ pointing inside this interval. We will confirm that the edges of $\Sat(\Phi)$ 
outside $[-N-1,N+1]$ are directed to this interval, and that the formula for $k$ is as suggested by~\eqref{kminmax}.
Based on these observations on $\Sat(\Phi)$, we show in Proposition~\ref{p:extension} that any solution $\Phi$
can be obtained as a unique extension of an eigenvector of the submatrix of $A$ extracted from $[-N-1,N+1]$. The 
description of solutions in terms of fundamental eigenvectors is obtained in Theorem~\ref{t:fundam}.

\begin{proposition}
\label{p:sat-prop}
Suppose that $U$ satisfies assumption (A$U$), $\Phi$ satisfies $A\otimes \Phi=\Phi$ and assumption
(A$\Phi$). Then
\begin{itemize}
\item[1.] $\Sat(\Phi)$ contains all backward edges $i\to i-1$ for $i>N$ and all
forward edges $i\to i+1$ for $i<-N$
\item[2.] In the restriction of $\Sat(\Phi)$ to $[-N-l,N+l]$ for $l\geq 1$, every node has an outgoing edge.
\item[3.] $k$ is given by
\begin{equation}
\label{kformula}
k=\min(\max_i(u_{i-1}+u_i),1).
\end{equation}
\end{itemize}
\end{proposition}
\begin{proof}
As (A$\Phi$) and (A$U$) are satisfied, there exist arbitrarily large
$N',N''\geq N$, such that 
\begin{equation}
\label{Nprime}
\begin{split}
\Phi_{i-1}&=-k+\Phi_i=A_{i-1,i}+\Phi_i,\quad\text{for $i=-N''$},\\
\Phi_{i+1}&=\Phi_i=A_{i+1,i}+\Phi_i,\quad\text{for $i=N'$}.
\end{split}
\end{equation}
Thus the edges $N'+1\to N'$ and $-N''-1\to -N''$
belong to $\Sat(\Phi)$.

First we have to treat the exceptional case $k=0$. In this case 
the cycles $(N',N'+1)$ and $(-N'',-N''-1)$ belong to $\Sat(\Phi)$. Hence 
by Theorem~\ref{l:satprops} they have the greatest
cycle weight in $\digr$, which is equal to $0$. In this case also all cycles for 
$(i-1,i)$ for $i\leq -N$ and $(i,i+1)$ for $i\geq N$ 
have weight $0$ and belong to $\Sat(\Phi)$ by Theorem~\ref{l:satprops}, hence part 1. 
Having $u_i+u_{i-1}>0$ is impossible in this case, as it leads to cycles
with a positive weight. If $u_i+u_{i-1}\leq 0$ for all $i$, then the greatest cycle mean is $\max(u_i+u_{i-1})=0$,
which equals $0$. Hence also part 3. For part 2, notice that there are only edges $i\to i+1$
and $i\to i-1$ in $\digr$, and that the cycles $(N+l,N+l+1)$ and $(-N-l,-N-l-1)$ belong to $\Sat(\Phi)$ for any $l\geq 0$. 

If $k>0$ then all cycles $(i,i+1)$ for $i\geq N$ and $(i,i-1)$ for $i\leq -N$
have weight $-k<0$ and they should not be in $\Sat(\Phi)$. In particular, $\Sat(\Phi)$
does not contain the edge $N'\to N'+1$, since it contains $N'+1\to N'$. 
However, $N'$ has an outgoing edge in $\Sat(\Phi)$, which
must be $N'\to N'-1$. Then $N'-1\to N'$ cannot be there if $N'-1\geq N$, so $\Sat(\Phi)$ contains
the edge $N'-1\to N'-2$ as well. Proceeding this way we obtain that $\Sat(\Phi)$ contains
all backward edges $i+1\to i$ for $N\leq i\leq N'$. Similarly $\Sat(\Phi)$ contains all
forward edges $i-1\to i$ for $-N''\leq i\leq -N$. Since $N'$ and $N''$ are arbitrarily 
large, part 1. follows. We also obtain that in the restriction of $\digr$ to $[-N,N]$ and more generally to
$[-N-i,N+i]$ where $i\geq 0$ there are no edges of $\Sat(\Phi)$ pointing outside of the interval, hence part~2.
Indeed, by the outgoing edge property, for each node in $[-N-i,N+i]$ there is an outgoing edge, which has to
point, by part 1., to another node in $[-N-i,N+i]$. 
For part 3,  Theorem~\ref{l:satprops} implies that the greatest total weight of a two-cycle does not
exceed zero.
It amounts to show that $\Sat(\Phi)$ contains cycles, which necessarily have zero total weight. 
This follows from part 2, since the restriction of $\Sat(\Phi)$ to $[-N-1,N+1]$ is finite
and each node has an outgoing edge. Hence the greatest total weight is zero, and \eqref{kformula} follows
from~\eqref{kminmax}.
\end{proof}

Part 1. shows that (A$\Phi$) implies the bound state condition when $k>0$. More precisely, it
implies that $\Phi_i=\Phi_{i+1}$ for all $i\geq N$ and $\Phi_{i-1}=\Phi_i-k$ for all
$i\leq -N$. When $k>0$, the condition $l\geq 1$ in part 2. can be replaced by $l\geq 0$.
In the same vein $[-N-1,N+1]$ can be replaced with $[-N,N]$ in the statements below, when $k>0$.

Denote by $\Phi_{[N_1,N_2]}$ the restriction of $\Phi$ to the interval $[N_1,N_2]$,
and by $A_{[N_1,N_2]}$ the submatrix extracted from the nodes in the interval 
$[N_1,N_2]$.

\begin{proposition}
\label{p:extension}
Suppose that $U$ satisfies assumption (A$U$) and $k\geq 0$. Then
\begin{itemize}
\item[1.] If $v$ satisfies $A_{[-N-1,N+1]}\otimes v=v$, then it can be uniquely extended to 
$\Phi$ satisfying (A$\Phi$) and $A\otimes\Phi=\Phi$, such that $\Phi_{[-N-1,N+1]}=v$.
\item[2.] If $\Phi$ satisfies $A\otimes\Phi=\Phi$, then $\Phi_{[-N-1,N+1]}$
satisfies $A_{[-N-1,N+1]}\otimes\Phi_{[-N-1,N+1]}=\Phi_{[-N-1,N+1]}$.
\end{itemize}
\end{proposition} 
\begin{proof}
1.: By  Proposition~\ref{p:sat-prop} part 1, every eigenvector satisfying (A$\Phi$), has to follow  $\Phi_{i+1}=\Phi_i$ for $i\geq N$
and $\Phi_i-k =\Phi_{i-1}$ for $i\leq -N$. Hence it is uniquely determined by $\Phi_{[-N-1,N+1]}$.
We also observe that $\Phi_{i+1}=\Phi_i$ implies $\Phi_i-k\leq\Phi_{i+1}$, and that $\Phi_i-k =\Phi_{i-1}$ implies
$\Phi_{i-1}\leq\Phi_i$ which makes it possible to extend $v$, satisfying $A_{[-N-1,N+1]}\otimes v=v$, to $\Phi$ which satisfies 
both $A\otimes\Phi=\Phi$ and (A$\Phi$).
 
2: If $A\otimes \Phi=\Phi$ then $A_{[N_1,N_2]}\otimes \Phi_{[N_1,N_2]}\leq \Phi_{[N_1,N_2]}$ for any $N_1,N_2$.
By Proposition~\ref{p:sat-prop} part 2, for each $i\in [-N-1,N+1]$ there is $j\in [-N-1,N+1]$ such that
$A_{ij}+\Phi_j=\Phi_i$, implying that $A_{[-N-1,N+1]}\otimes \Phi_{[-N-1,N+1]}\geq \Phi_{[-N-1,N+1]}$. Combining with the 
reverse inequality, we obtain part 2.
\end{proof}


For a (possibly infinite-dimensional) matrix $A$, the {\em Kleene star} is introduced as in~\eqref{kls-0} by
$$
A^*:=I\oplus A\oplus A^2\oplus\ldots.
$$
In the infinite-dimensional case, this may have infinite number of terms. However, in our case
the number of terms is always finite for any entry of the Kleene star, and the weight of the entry $A^*_{ij}$
equals to the greatest total weight (i.e., sum of weights of the edges) among all paths connecting $i$ to $j$.

As in the introduction, by the {\em critical columns} of $A^*$ we understand the columns of $A^*$ with indices taken from the
critical graph of $A$.

\begin{theorem}
\label{t:fundam}
Let $U$ satisfy (A$U$). The set of eigenvectors $A\otimes \Phi=\Phi$ satisfying (A$\Phi$) is nonempty if and only
if $k$ is given by~\eqref{kformula}. In this case it is the set of max-linear combinations
of the critical columns of $A^*$, which can be also computed as unique extensions of max-linear combinations
with the same coefficients, of the columns of $(A_{[-N-1,N+1]})^*$ with the same indices.  
\end{theorem}
\begin{proof}
First note that $(A_{[-N-1,N+1]})^*=(A^*)_{[-N+1,N+1]}$. Indeed, for $i,j\in[-N-1,N+1]$, if a path connecting $i$ to $j$
has nodes outside $[-N-1,N+1]$ then it contains cycles. These cycles can be cancelled preserving connectivity of the path and
not decreasing its total weight, until all nodes of the path are in $[-N-1,N+1]$. Thus, for any such $i,j$ there exists an optimal path
which is entirely in $[-N-1,N+1]$. Then $(A_{[-N-1,N+1]})^*=(A^*)_{[-N+1,N+1]}$ follows 
entrywise by the path interpretation of Kleene star. 

Next, it can be verified that any column of $A^*$ with an index in the critical graph satisfies (A$\Phi$), and then so does any max-linear combination of these columns.

Any eigenvector $v$ of $A_{[-N-1,N+1]}$ is a max-linear combination of the columns of $(A_{[-N-1,N+1]})^*$ with indices
in the critical graph, see Theorem~\ref{t:max-spectrum}. Since $(A_{[-N-1,N+1]})^*=(A^*)_{[-N+1,N+1]}$, the max-linear combination with the same coefficients of the columns of $A^*$ with the same indices, is an extension of $v$ satisfying (A$\Phi$). Such
extension is unique by Proposition~\ref{p:extension} part 1.  Conversely, by Proposition~\ref{p:extension} part 2,
any vector satisfying $A\otimes\Phi=\Phi$ is a unique extension of an eigenvector of $A_{[-N-1,N+1]}$.
\end{proof}

\subsection{Undressing transform}

In this subsection, assumptions (A$\Phi$) and (A$U$) are assumed everywhere.

It is easy to see that each soliton corresponds to a critical component in
the associated digraphs of $A(\gamma)$ and $A(\delta)$. Hence it follows that
each soliton gives rise to a pair of fundamental eigenvectors of $A(\gamma)$ and
$A(\delta)$ which we denote by $\Phi^{(1)}$ and $\Phi^{(2)}$. We next examine the transformation $U\mapsto\Tilde{U}$ defined by
\begin{equation}
\label{e:undress}
\Tilde{u}_i:=u_i+\Phi^{(1)}_{i+1}+\Phi^{(2)}_i-\Phi^{(1)}_i-\Phi^{(2)}_{i+1}.
\end{equation}

To compute this transformation explicitly, we 
need the relations between neighbouring coordinates of $\Phi^{(1)}$ and $\Phi^{(2)}$.
They are as follows.

\begin{proposition}
\label{phirel-s1}
Suppose that (C1) holds. Let $(l,\ldots,l+s)$ be a soliton and let
$\Phi^{(1)}$ and $\Phi^{(2)}$ be the pair of fundamental eigenvectors
of $A(\gamma)$, resp. $A(\delta)$ associated with it. Then
\begin{equation}
\label{e:phirel-s1}
\begin{split}
&\Phi_i^{(1)}-\Phi_{i+1}^{(1)}=u_i-k,\; i\leq l,\\
&\Phi_{i+1}^{(1)}-\Phi_i^{(1)}=u_{i+1},\; i\geq l,\\
& \Phi_i^{(2)}-\Phi_{i+1}^{(2)}=u_{i-1}-k,\; i\leq l+s-1,\\
& \Phi_{i+1}^{(2)}-\Phi_i^{(2)}=u_i,\; i\geq l+s-1.
\end{split}
\end{equation}
\end{proposition}

\begin{proposition}
\label{phirel-s2}
Suppose that (C2) holds. Let $(l,\ldots,l+s)$ be a soliton and let
$\Phi^{(1)}$ and $\Phi^{(2)}$ be the pair of fundamental eigenvectors
of $A(\gamma)$, resp. $A(\delta)$ associated with it. Then
\begin{equation}
\label{e:phirel-s2}
\begin{split}
&\Phi_i^{(1)}-\Phi_{i+1}^{(1)}=
\begin{cases}
u_i-1, &\text{if $u_i+u_{i-1}<1$},\\
-u_{i-1}, &\text{if $u_i+u_{i-1}\geq 1$}
\end{cases},
\; i\leq l,\\ 
&\Phi_{i+1}^{(1)}-\Phi_i^{(1)}=
\begin{cases}
u_{i+1}, &\text{if $u_i+u_{i+1}<1$},\\
1-u_i, &\text{if $u_i+u_{i+1}\geq 1$}
\end{cases},
\; i\geq l,\\ 
&\Phi_{i}^{(2)}-\Phi_{i+1}^{(2)}=
\begin{cases}
u_{i-1}-1, &\text{if $u_i+u_{i-1}<1$},\\
-u_i, &\text{if $u_i+u_{i-1}\geq 1$}
\end{cases},
\; i\leq l+s-1,\\
&\Phi_{i+1}^{(2)}-\Phi_i^{(2)}=
\begin{cases}
u_i, &\text{if $u_i+u_{i+1}<1$},\\
1-u_{i+1}, &\text{if $u_i+u_{i+1}\geq 1$}
\end{cases},
\; i\geq l+s-1.
\end{split}
\end{equation}
\end{proposition}

\begin{proof} (Propositions~\ref{phirel-s1} and~\ref{phirel-s2}) 
In both cases, we essentially have to examine which edges of $A(\gamma)$
and $A(\delta)$ are in the saturation graphs of $\Phi^{(1)}$ and
$\Phi^{(2)}$. Then we use the explicit formulas for the coefficients of
$A(\gamma)$ and $A(\delta)$, see Lemmas~\ref{coeffs-s1} and~\ref{coeffs-s2}. 

In both cases (C1) and (C2), the cycle $(l,l+1)$ is critical in $A(\gamma)$, thus $\Phi^{(1)}$
can be chosen as $l$ or $l+1$ column of $A(\gamma)^*$.  If $j\leq l$ then 
\begin{equation}
\label{difference1}
\Phi^{(1)}_j-\Phi^{(1)}_{l+1}=A(\gamma)^*_{j,l+1}= \sum_{i=j}^{l} A(\gamma)_{i,i+1}.
\end{equation}
Hence $\Phi^{(1)}_i-\Phi^{(1)}_{i+1}=A(\gamma)_{i,i+1}$ for all $i\leq l$.
Analogously for $j\geq l$
\begin{equation}
\label{difference2}
\Phi^{(1)}_j-\Phi^{(1)}_l=A(\gamma)^*_{j,l}= \sum_{i=l+1}^{j} A(\gamma)_{i,i-1}.
\end{equation}
Hence $\Phi^{(1)}_{i+1}-\Phi^{(1)}_i=A(\gamma)_{i+1,i}$ for all $i\geq l$.

In both cases (C1) and (C2),
the cycle $(l+s-1,l+s)$ is critical in $A(\delta)$,  
thus $\Phi^{(2)}$
can be chosen as $l+s-1$ or $l+s$ column of $A(\delta)^*$.
Arguing as above, we obtain that
$\Phi^{(2)}_i-\Phi^{(2)}_{i+1}=A(\delta)_{i,i+1}$ for all $i\leq l+s-1$, and that
$\Phi^{(2)}_{i+1}-\Phi^{(2)}_i=A(\delta)_{i+1,i}$ for all $i\geq l+s-1$.
Summarizing we have:
\begin{equation}
\label{e:phirel-s2-simple}
\begin{split}
&\Phi^{(1)}_i-\Phi^{(1)}_{i+1}=A(\gamma)_{i,i+1},\quad i\leq l,\\
&\Phi^{(1)}_{i+1}-\Phi^{(1)}_i=A(\gamma)_{i+1,i},\quad i\geq l,\\
&\Phi^{(2)}_i-\Phi^{(2)}_{i+1}=A(\delta)_{i,i+1},\quad i\leq l+s-1,\\
&\Phi^{(2)}_{i+1}-\Phi^{(2)}_i=A(\delta)_{i+1,i},\quad i\geq l+s-1.
\end{split}
\end{equation}
It remains to use the explicit expressions for coefficients of $A(\gamma)$ and $A(\delta)$.
\end{proof}

Next we establish explicit expressions for undressing~\eqref{e:undress}, in the situations (C1) and (C2).

\begin{theorem}
\label{t:undress}
Let $(l,\ldots,l+s)$ be a soliton and let
$\Phi^{(1)}$ and $\Phi^{(2)}$ be the pair of fundamental eigenvectors
of $A(\gamma)$, resp. $A(\delta)$ associated with it. Then in the case of
(C1)
\begin{equation}
\label{e:undress1}
\Tilde{u}_i=
\begin{cases}
u_{i-1}, &\text{if $i\leq l$},\\
u_{i-1}=u_{i+1}, &\text{if $l<i<l+s$},\\
u_{i+1}, &\text{if $i\geq l+s$},
\end{cases}
\end{equation}
and in the case of (C2)
\begin{equation}
\label{e:undress2}
\Tilde{u}_i=
\begin{cases}
u_{i-1}, &\text{if $i\leq l$},\\
1-u_i, &\text{if $l<i<l+s$},\\
u_{i+1}, &\text{if $i\geq l+s$}.
\end{cases}
\end{equation}
\end{theorem}
\begin{proof}
The computation is straightforward, using~\eqref{e:undress} and \eqref{e:phirel-s1} 
in the case of (C1), or~\eqref{e:undress} and \eqref{e:phirel-s2} in the case of (C2).

Namely in the case $i\leq l$ we use the first and the third relations of~\eqref{e:phirel-s1} and~\eqref{e:phirel-s2}, substituting
them into~\eqref{e:undress}. In the case $i\geq l+s$ we use the second and the fourth relations
of~\eqref{e:phirel-s1} and~\eqref{e:phirel-s2}. 

In the case $l<i<l+s$, we use the second and the
third relations of \eqref{e:phirel-s1} and~\eqref{e:phirel-s2}. If (C1) holds, note that for $l<i<l+s$ the computation yields
$u_{i+1}+u_{i-1}+u_i-k$. However all cycles
$(i,i-1)$ and $(i,i+1)$ are critical with $u_{i+1}+u_i-k=u_i+u_{i-1}-k=0$, hence we obtain $u_{i-1}=u_{i+1}$ as in~\eqref{e:undress1}. 
If (C2) holds, then the computation (use the case $u_i+u_{i+1}\geq 1$ since we are inside the soliton) 
yields $2-u_i-k$, which is $1-u_i$ since $k=1$.
\end{proof}

We obtain that in the case of (C1), according to~\eqref{e:undress1},
the selected soliton $(l,\ldots,l+s)$ loses two units of its length, and the rest of the potential $U$ gets shifted to the right before the selected soliton, and to the left after the selected soliton. The remaining part of the selected soliton
also gets shifted, but the direction does not matter since $u_{i-1}=u_{i+1}$ for all $l<i<l+s$.

In the case of (C2), according to~\eqref{e:undress2}, 
the selected soliton also loses at least $2$ units of length on the ends, 
and the remaining part may shrink and brake into several solitons. Indeed, we have $$\Tilde{u}_l+\Tilde{u}_{l+1}=
1-u_{l+1}+u_{l-1}\leq u_l+u_{l-1}<1$$ and also
$$\Tilde{u}_{l+s-1}+\Tilde{u}_{l+s}=
1-u_{l+s-1}+u_{l+s+1}\leq u_{l+s}+u_{l+s+1}<1.$$
For $i$ in $l<i<l+s-1$ we obtain
$$\Tilde{u}_i+\Tilde{u}_{i+1}=2-u_i-u_{i+1},$$
which is not less than $1$ only if $u_i+u_{i+1}=1$.

An important special case of solitons in case (C2) is when they are of the form $(a\; 1\ldots 1\;b)$ with $a,b\geq 0$, and when all
elements between them and outside the soliton area equal $0$. In this case, such soliton completely disappears 
turning into $0$ background after the corresponding undressing transform (which justifies the name ``undressing'').

We give a graphical example of undressing where the selected soliton
has length 3. Figure~\ref{f:undr-C1} demonstrates undressing in the case (C1): look at the difference
between $U$ and $\Tilde{U}$. Figure~\ref{f:undr-C2} demonstrates undressing in the case (C2).

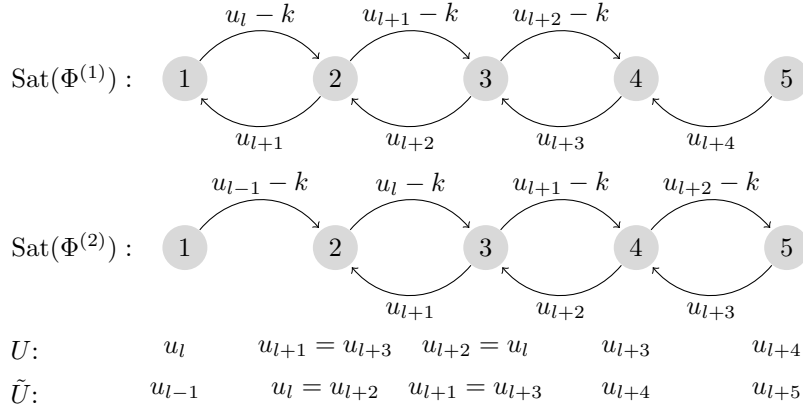
\begin{figure}
\centering
\begin{tikzpicture}[shorten >=1pt,->]
  \tikzstyle{vertex}=[circle,fill=black!15,minimum size=17pt,inner sep=1pt]
\foreach \name/\x in {1/3, 2/5, 3/7, 4/9, 5/11}
    \node[vertex] (\name) at (\x,0) {$\name$};
\draw[xshift=1.5cm] node{$\Sat(\Phi^{(1)}):$};

\draw (1) .. controls +(50:1cm) and +(130:1cm) .. (2);
\draw (2) .. controls +(50:1cm) and +(130:1cm) .. (3);
\draw (3) .. controls +(50:1cm) and +(130:1cm) .. (4);

\draw[xshift=4cm,yshift=0.85cm] node{$u_l-k$};
\draw[xshift=6cm,yshift=0.85cm] node{$u_{l+1}-k$};
\draw[xshift=8cm,yshift=0.85cm] node{$u_{l+2}-k$};


\draw (5) .. controls +(-130:1cm) and +(-50:1cm) .. (4);
\draw (4) .. controls +(-130:1cm) and +(-50:1cm) .. (3);
\draw (3) .. controls +(-130:1cm) and +(-50:1cm) .. (2);
\draw (2) .. controls +(-130:1cm) and +(-50:1cm) .. (1);

\draw[xshift=4cm,yshift=-0.85cm] node{$u_{l+1}$};
\draw[xshift=6cm,yshift=-0.85cm] node{$u_{l+2}$};
\draw[xshift=8cm,yshift=-0.85cm] node{$u_{l+3}$};
\draw[xshift=10cm,yshift=-0.85cm] node{$u_{l+4}$};

\end{tikzpicture}

\begin{tikzpicture}[shorten >=1pt,->]
  \tikzstyle{vertex}=[circle,fill=black!15,minimum size=17pt,inner sep=1pt]
\foreach \name/\x in { 1/3, 2/5, 3/7, 4/9, 5/11}
    \node[vertex] (\name) at (\x,0) {$\name$};
\draw[xshift=1.5cm] node{$\Sat(\Phi^{(2)}):$};

\draw (1) .. controls +(50:1cm) and +(130:1cm) .. (2);
\draw (2) .. controls +(50:1cm) and +(130:1cm) .. (3);
\draw (3) .. controls +(50:1cm) and +(130:1cm) .. (4);
\draw (4) .. controls +(50:1cm) and +(130:1cm) .. (5);

\draw[xshift=4cm,yshift=0.85cm] node{$u_{l-1}-k$};
\draw[xshift=6cm,yshift=0.85cm] node{$u_l-k$};
\draw[xshift=8cm,yshift=0.85cm] node{$u_{l+1} -k$};
\draw[xshift=10cm,yshift=0.85cm] node{$u_{l+2}-k$};

\draw (5) .. controls +(-130:1cm) and +(-50:1cm) .. (4);
\draw (4) .. controls +(-130:1cm) and +(-50:1cm) .. (3);
\draw (3) .. controls +(-130:1cm) and +(-50:1cm) .. (2);

\draw[xshift=6cm,yshift=-0.85cm] node{$u_{l+1}$};
\draw[xshift=8cm,yshift=-0.85cm] node{$u_{l+2}$};
\draw[xshift=10cm,yshift=-0.85cm] node{$u_{l+3}$};

\end{tikzpicture}

\begin{tikzpicture}[shorten >=1pt,->]

\draw[xshift=1.5cm] node{$U$:};

\draw[xshift=3.5cm,yshift=0cm] node{$u_{l}$};
\draw[xshift=5.5cm,yshift=0cm] node{$u_{l+1}=u_{l+3}$};
\draw[xshift=7.5cm,yshift=0cm] node{$u_{l+2}=u_l$};
\draw[xshift=9.5cm,yshift=0cm] node{$u_{l+3}$};
\draw[xshift=11.5cm,yshift=0cm] node{$u_{l+4}$};

\end{tikzpicture}

\begin{tikzpicture}[shorten >=1pt,->]

\draw[xshift=1.5cm] node{$\Tilde{U}$:};

\draw[xshift=3.5cm,yshift=0cm] node{$u_{l-1}$};
\draw[xshift=5.5cm,yshift=0cm] node{$u_{l}=u_{l+2}$};
\draw[xshift=7.5cm,yshift=0cm] node{$u_{l+1}=u_{l+3}$};
\draw[xshift=9.5cm,yshift=0cm] node{$u_{l+4}$};
\draw[xshift=11.5cm,yshift=0cm] node{$u_{l+5}$};

\end{tikzpicture}
\caption{\label{f:undr-C1} Undressing in (C1)}
\end{figure}

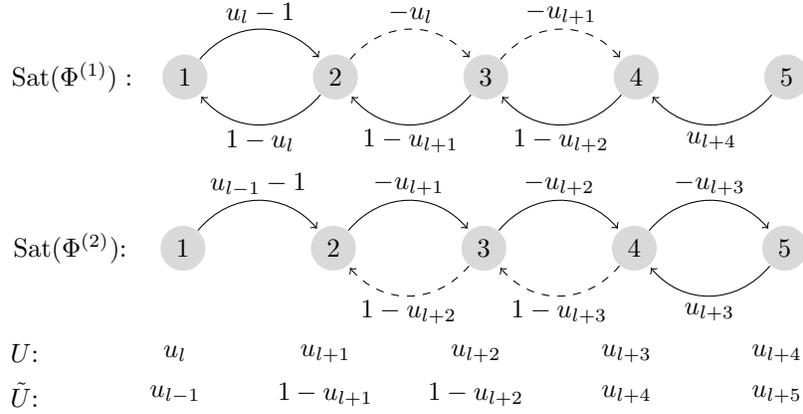
\begin{figure}
\centering
\begin{tikzpicture}[shorten >=1pt,->]
  \tikzstyle{vertex}=[circle,fill=black!15,minimum size=17pt,inner sep=1pt]
\foreach \name/\x in {1/3, 2/5, 3/7, 4/9, 5/11}
    \node[vertex] (\name) at (\x,0) {$\name$};
\draw[xshift=1.5cm] node{$\Sat(\Phi^{(1)}):$};

\draw (1) .. controls +(50:1cm) and +(130:1cm) .. (2);
\begin{scope}[dashed]
\draw (2) .. controls +(50:1cm) and +(130:1cm) .. (3);
\draw (3) .. controls +(50:1cm) and +(130:1cm) .. (4);
\end{scope}

\draw[xshift=4cm,yshift=0.85cm] node{$u_l-1$};
\draw[xshift=6cm,yshift=0.85cm] node{$-u_l$};
\draw[xshift=8cm,yshift=0.85cm] node{$-u_{l+1}$};


\draw (5) .. controls +(-130:1cm) and +(-50:1cm) .. (4);
\draw (4) .. controls +(-130:1cm) and +(-50:1cm) .. (3);
\draw (3) .. controls +(-130:1cm) and +(-50:1cm) .. (2);
\draw (2) .. controls +(-130:1cm) and +(-50:1cm) .. (1);

\draw[xshift=4cm,yshift=-0.85cm] node{$1-u_l$};
\draw[xshift=6cm,yshift=-0.85cm] node{$1-u_{l+1}$};
\draw[xshift=8cm,yshift=-0.85cm] node{$1-u_{l+2}$};
\draw[xshift=10cm,yshift=-0.85cm] node{$u_{l+4}$};

\end{tikzpicture}

\begin{tikzpicture}[shorten >=1pt,->]
  \tikzstyle{vertex}=[circle,fill=black!15,minimum size=17pt,inner sep=1pt]
\foreach \name/\x in { 1/3, 2/5, 3/7, 4/9, 5/11}
    \node[vertex] (\name) at (\x,0) {$\name$};
\draw[xshift=1.5cm] node{$\Sat(\Phi^{(2)})$:};

\draw (1) .. controls +(50:1cm) and +(130:1cm) .. (2);
\draw (2) .. controls +(50:1cm) and +(130:1cm) .. (3);
\draw (3) .. controls +(50:1cm) and +(130:1cm) .. (4);
\draw (4) .. controls +(50:1cm) and +(130:1cm) .. (5);

\draw[xshift=4cm,yshift=0.85cm] node{$u_{l-1}-1$};
\draw[xshift=6cm,yshift=0.85cm] node{$-u_{l+1}$};
\draw[xshift=8cm,yshift=0.85cm] node{$-u_{l+2}$};
\draw[xshift=10cm,yshift=0.85cm] node{$-u_{l+3}$};

\draw (5) .. controls +(-130:1cm) and +(-50:1cm) .. (4);
\begin{scope}[dashed]
\draw (4) .. controls +(-130:1cm) and +(-50:1cm) .. (3);
\draw (3) .. controls +(-130:1cm) and +(-50:1cm) .. (2);
\end{scope}

\draw[xshift=6cm,yshift=-0.85cm] node{$1-u_{l+2}$};
\draw[xshift=8cm,yshift=-0.85cm] node{$1-u_{l+3}$};

\draw[xshift=10cm,yshift=-0.85cm] node{$u_{l+3}$};

\end{tikzpicture}

\begin{tikzpicture}[shorten >=1pt,->]

\draw[xshift=1.5cm] node{$U$:};

\draw[xshift=3.5cm,yshift=0cm] node{$u_{l}$};
\draw[xshift=5.5cm,yshift=0cm] node{$u_{l+1}$};
\draw[xshift=7.5cm,yshift=0cm] node{$u_{l+2}$};
\draw[xshift=9.5cm,yshift=0cm] node{$u_{l+3}$};
\draw[xshift=11.5cm,yshift=0cm] node{$u_{l+4}$};

\end{tikzpicture}

\begin{tikzpicture}[shorten >=1pt,->]

\draw[xshift=1.5cm] node{$\Tilde{U}$:};

\draw[xshift=3.5cm,yshift=0cm] node{$u_{l-1}$};
\draw[xshift=5.5cm,yshift=0cm] node{$1-u_{l+1}$};
\draw[xshift=7.5cm,yshift=0cm] node{$1-u_{l+2}$};
\draw[xshift=9.5cm,yshift=0cm] node{$u_{l+4}$};
\draw[xshift=11.5cm,yshift=0cm] node{$u_{l+5}$};

\end{tikzpicture}
\caption{\label{f:undr-C2}
Undressing in (C2). The dashed edges indicate that they belong to saturation graphs if and only if $u_i+u_{i+1}=1$ for
the corresponding $i$, that is, if the two-cycle containing them is critical.}
\end{figure}

\section{Adding constraints}

In this section we verify whether a fundamental pair 
$\Phi^{(1)}$, $\Phi^{(2)}$ also satisfies
the last two equations of~\eqref{e:willox}, which we rewrite as 
\begin{equation}
\label{backward}
\Phi_l^{(1)} =\max(\Phi_{l+1}^{(2)},\Phi^{(1)}_{l+1}+u_l-1),\\
\end{equation}

\begin{equation}
\label{forward}
\Phi_{l+1}^{(2)} =\max(\Phi_l^{(1)}-\mys,\Phi_l^{(2)}+u_l+k-1).
\end{equation}
The parameter $\mys$, equal to $\omega-k$ in~\eqref{e:willox}, will be specified later.
Let us remark so far, that since $\Phi_l^{(1)}=\Phi_l^{(2)}=0$ at all large enough $l$,
using~\eqref{forward} we obtain $\mys\geq 0$.

\subsection{Case (C1)}

\begin{proposition}
\label{p:C1}
In case (C1), any fundamental eigenpair satisfies~\eqref{backward} and~\eqref{forward}.
\end{proposition}
\begin{proof}
In the case $(C1)$, the graph $\digr(\delta)$ is the same as the graph $\digr(\gamma)$ shifted
one position to the right, and the same is true about the graphs $\Sat(\Phi^{(1)})$ and $\Sat(\Phi^{(2)})$ for any
fundamental eigenpair $\Phi^{(1)},\Phi^{(2)}$. Then $\Phi^{(1)}_{i+1}-\Phi^{(1)}_i=\Phi^{(2)}_{i+2}-\Phi^{(2)}_{i+1}$ 
for all $i$ and hence $\Phi_i^{(1)}=\Phi_{i+1}^{(2)}$ for all $i$. We put $\mys=0$ and verify the
remaining inequalities 
\begin{equation}
\begin{split}
&\Phi_i^{(1)}\geq\Phi^{(1)}_{i+1}+u_i-1\\
&\Phi_{i+1}^{(2)}\geq \Phi_i^{(2)}+u_i+k-1
\end{split}
\end{equation}
comparing them with~\eqref{e:phirel-s1}. The verification follows from $k\leq 1$ and $u_i+u_{i+1}\leq 1$. 
\end{proof}

Thus in the case (C1) any fundamental eigenpair satisfies~\eqref{e:willox}.
Also note that due to max-plus linearity, any max-plus combination of fundamental eigenpairs
$(\Phi^{(1)},\Phi^{(2)})$ is again a solution of the system, so that in general the solution space
is highly degenerated. Note that this result justifies the study of undressing by means of fundamental pairs in
the case (C1). 

\subsection{Case (C2), one soliton}

Here we verify that in the case when there is just one soliton in case (C2),
$(l,\ldots,l+s)$, the fundamental pair satisfies~\eqref{backward} and~\eqref{forward}.

In the following table, we consider an example where the potential (i.e., solution of udKdV) consists of
one soliton $(\pi_1\; \pi_2\;\pi_3\;\pi_4)$, where the real numbers $\pi_1,\pi_2,\pi_3,\pi_4<1$ are 
real numbers such that $\pi_1+\pi_2>1,$ $\pi_2+\pi_3>1$ and $\pi_3+\pi_4>1$.

\begin{equation*}
\begin{array}{cccccccc}
l:	 & 0 & 1 & 2 & 3 & 4 & 5 & 6\\      	
u_l:       & 0 & \pi_1 & \pi_2 & \pi_3 & \pi_4 & 0 & 0\\
\gamma_l=\min(u_l,1-u_{l-1}):  & 0 & \pi_1 & 1-\pi_1  &  1-\pi_2 & 1-\pi_3 & 0     & 0\\
\delta_l=\min(u_{l-1},1-u_l):  & 0 & 0    & 1-\pi_2  & 1-\pi_3  & 1-\pi_4 & \pi_4 & 0   
\end{array}
\end{equation*}

The digraphs $\digr(\gamma)$ and $\digr(\delta)$ are displayed on Figure~\ref{f:onesoliton}.
\begin{figure}
\centering
\begin{tikzpicture}[shorten >=1pt,->, scale=0.8]
  \tikzstyle{vertex1}=[circle,fill=black!15,minimum size=17pt,inner sep=1pt]
\tikzstyle{vertex2}=[circle,fill=black!30,minimum size=17pt,inner sep=1pt]

\foreach \name/\x in {0/1, 3/7, 4/9, 5/11, 6/13}
    \node[vertex1] (\name) at (\x,0) {$\name$};
\foreach \name/\x in {1/3, 2/5}
    \node[vertex2] (\name) at (\x,0) {$\name$};

\foreach \name/\nickname\x in {0'/0/1, 1'/1/3, 2'/2/5, 3'/3/7, 6'/6/13}
    \node[vertex1] (\name) at (\x,-3) {$\nickname$};
\foreach \name/\nickname/\x in {4'/4/9, 5'/5/11}
    \node[vertex2] (\name) at (\x,-3) {$\nickname$};

\draw[xshift=1cm, yshift=-5cm] node{$0$:};
\draw[xshift=3cm,yshift=-5cm] node{$\pi_1$};
\draw[xshift=5cm,yshift=-5cm] node{$\pi_2$};
\draw[xshift=7cm,yshift=-5cm] node{$\pi_3$};
\draw[xshift=9cm,yshift=-5cm] node{$\pi_4$};
\draw[xshift=11cm,yshift=-5cm] node{$0$};
\draw[xshift=13cm,yshift=-5cm] node{$0$};




\draw (0) .. controls +(50:1cm) and +(130:1cm) .. (1);
\draw (1) .. controls +(50:1cm) and +(130:1cm) .. (2);
\begin{scope}[dashed]
\draw (2) .. controls +(50:1cm) and +(130:1cm) .. (3);
\draw (3) .. controls +(50:1cm) and +(130:1cm) .. (4);
\draw (4) .. controls +(50:1cm) and +(130:1cm) .. (5);
\draw (5) .. controls +(50:1cm) and +(130:1cm) .. (6);
\end{scope}

\draw[xshift=2cm,yshift=0.85cm] node{$-1$};
\draw[xshift=4cm,yshift=0.85cm] node{$\pi_1-1$};
\draw[xshift=6cm,yshift=0.85cm] node{$-\pi_1$};
\draw[xshift=8cm,yshift=0.85cm] node{$-\pi_2$};
\draw[xshift=10cm,yshift=0.85cm] node{$-\pi_3$};
\draw[xshift=12cm,yshift=0.85cm] node{$-1$};

\draw (6) .. controls +(-130:1cm) and +(-50:1cm) .. (5);
\draw (5) .. controls +(-130:1cm) and +(-50:1cm) .. (4);
\draw (4) .. controls +(-130:1cm) and +(-50:1cm) .. (3);
\draw (3) .. controls +(-130:1cm) and +(-50:1cm) .. (2);
\draw (2) .. controls +(-130:1cm) and +(-50:1cm) .. (1);
\begin{scope}[dashed]
\draw (1) .. controls +(-130:1cm) and +(-50:1cm) .. (0);
\end{scope}

\draw[xshift=2cm,yshift=-0.85cm] node{$\pi_1$};
\draw[xshift=4cm,yshift=-0.85cm] node{$1-\pi_1$};
\draw[xshift=6cm,yshift=-0.85cm] node{$1-\pi_2$};
\draw[xshift=8cm,yshift=-0.85cm] node{$1-\pi_3$};
\draw[xshift=10cm,yshift=-0.85cm] node{$0$};
\draw[xshift=12cm,yshift=-0.85cm] node{$0$};



\draw (0') .. controls +(50:1cm) and +(130:1cm) .. (1');
\draw (1') .. controls +(50:1cm) and +(130:1cm) .. (2');
\draw (2') .. controls +(50:1cm) and +(130:1cm) .. (3');
\draw (3') .. controls +(50:1cm) and +(130:1cm) .. (4');
\draw (4') .. controls +(50:1cm) and +(130:1cm) .. (5');
\begin{scope}[dashed]
\draw (5') .. controls +(50:1cm) and +(130:1cm) .. (6');
\end{scope}

\draw[xshift=2cm,yshift=-2.15cm] node{$-1$};
\draw[xshift=4cm,yshift=-2.15cm] node{$-1$};
\draw[xshift=6cm,yshift=-2.15cm] node{$-\pi_2$};
\draw[xshift=8cm,yshift=-2.15cm] node{$-\pi_3$};
\draw[xshift=10cm,yshift=-2.15cm] node{$-\pi_4$};
\draw[xshift=12cm,yshift=-2.15cm] node{$\pi_4-1$};

\draw (6') .. controls +(-130:1cm) and +(-50:1cm) .. (5');
\draw (5') .. controls +(-130:1cm) and +(-50:1cm) .. (4');
\begin{scope}[dashed]
\draw (4') .. controls +(-130:1cm) and +(-50:1cm) .. (3');
\draw (3') .. controls +(-130:1cm) and +(-50:1cm) .. (2');
\draw (2') .. controls +(-130:1cm) and +(-50:1cm) .. (1');
\draw (1') .. controls +(-130:1cm) and +(-50:1cm) .. (0');
\end{scope}

\draw[xshift=2cm,yshift=-3.85cm] node{$0$};
\draw[xshift=4cm,yshift=-3.85cm] node{$1-\pi_2$};
\draw[xshift=6cm,yshift=-3.85cm] node{$1-\pi_3$};
\draw[xshift=8cm,yshift=-3.85cm] node{$1-\pi_4$};
\draw[xshift=10cm,yshift=-3.85cm] node{$\pi_4$};
\draw[xshift=12cm,yshift=-3.85cm] node{$0$};

\draw[dashed,-] (1) -- (2');
\draw[dashed,-] (4) -- (5');

\draw[xshift=7cm,yshift=-1.5cm] node{soliton area};
\draw[xshift=1cm,yshift=-1.5cm] node{before soliton};
\draw[xshift=12cm,yshift=-1.5cm] node{after soliton};
\end{tikzpicture}
\caption{\label{f:onesoliton}The case of one soliton}
\end{figure}
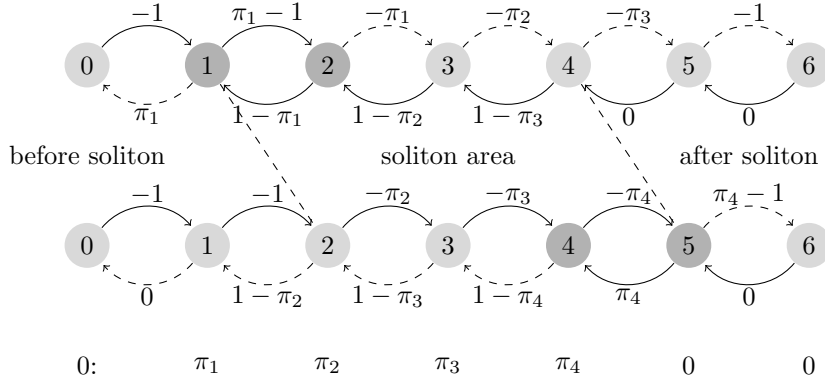

The saturation graph $\Sat(\Phi^{(1)})$ is a subgraph of $\digr(\gamma)$ shown in the upper part of the
picture: the edges not belonging to it are drawn as dashed.  Similarly, the saturation graph $\Sat(\Phi^{(2)})$ 
is a subgraph of $\digr(\delta)$ shown in the upper part of the picture. Combining these graphs, we see that
we have to analyse three cases: 1) before the soliton area (to the left), 2) in the soliton area, 3) after the
soliton area (to the right), with three different combinations of edges used by the fundamental pair.

To formalize the arguments let us introduce the notion of soliton area in general. 
Let $(l,\ldots,l+s)$ be a soliton. By the soliton area we mean a pair of
subgraphs of $\digr(\gamma)$ and $\digr(\delta)$: 1) the subgraph of 
$\digr(\gamma)$ extracted from the nodes $l,\ldots,l+s$, 
2) the subgraph of $\digr(\delta)$ extracted from the nodes $l+1,\ldots,l+s+1$.

\begin{theorem}
If $U$ contains just one soliton $(l,\ldots,l+s)$, then
the pair of fundamental eigenvectors associated with it 
satisfies~\eqref{backward} and~\eqref{forward}.
\end{theorem}
\begin{proof}
We are going to use relations between the neighbouring components of $\Phi^{(1)}$ and $\Phi^{(2)}$, written above 
in~\eqref{e:phirel-s2}. 

For the area after the soliton we obtain
\begin{equation}
\label{after-soliton}
\begin{split}
& \Phi_{i+1}^{(1)}-\Phi_i^{(1)}=u_{i+1},\\ 
& \Phi_{i+1}^{(2)}-\Phi_i^{(2)}=u_i, 
\end{split}
\end{equation}
and in particular $\Phi_{i+1}^{(1)}-\Phi_i^{(1)}=\Phi_{i+2}^{(2)}-\Phi_{i+1}^{(2)}$ for all $i\geq l+s$.
This implies $\Phi_i^{(1)}=\Phi_{i+1}^{(2)}$ for the area after the soliton. Equation~\eqref{forward} follows
from $\Phi_{i+1}^{(2)}\geq \Phi_i^{(1)}-\mys$, and since $\Phi_{i+1}^{(2)}=\Phi_i^{(2)}+u_i$ by the
second equation of~\eqref{after-soliton}. We also obtain
\begin{equation}
\Phi_i^{(1)}-\Phi_{i+1}^{(1)}\geq u_i-1
\end{equation}
from the first equation of~\eqref{after-soliton}, 
since $u_i+u_{i+1}\leq 1$ for $i\geq l+s$, which together with $\Phi_i^{(1)}=\Phi_{i+1}^{(2)}$ 
makes~\eqref{backward}. 

For the soliton area we obtain
\begin{equation}
\label{soliton-area}
\begin{split}
& \Phi_{i+1}^{(1)}-\Phi_i^{(1)}=1-u_i,\\ 
& \Phi_i^{(2)}-\Phi_{i+1}^{(2)}=-u_i, 
\end{split}
\end{equation}
and in particular $\Phi_i^{(1)}-\Phi_{i+1}^{(1)}\geq\Phi_{i+1}^{(2)}-\Phi_{i+2}^{(2)}$.
Equation~\eqref{soliton-area} implies that
\begin{equation}
\label{e:difference}
\Phi_{i-1}^{(1)}-\Phi_i^{(2)}=\Phi_i^{(1)}-\Phi_{i+1}^{(2)}+(u_i+u_{i-1}-1).
\end{equation}
Defining
\begin{equation}
\label{def:mys}
\mys=\sum_{i=l}^{l+s-1} (u_i+u_{i+1}-1)
\end{equation}
we obtain that
\begin{equation}
\Phi_i^{(1)}-\mys\leq\Phi_{i+1}^{(2)}\leq\Phi_i^{(1)}
\end{equation}
To show~\eqref{forward} and~\eqref{backward} we observe that~\eqref{soliton-area} furnish the remaining
necessary equalities $\Phi_{i+1}^{(2)}=\Phi_i^{(2)}+u_i$ and $\Phi_i^{(1)}=\Phi_{i+1}^{(1)}+u_i-1$.
 
Before the soliton area we obtain 
\begin{equation}
\label{before-soliton}
\begin{split}
& \Phi_{i}^{(1)}-\Phi_{i+1}^{(1)}=u_i-1,\\ 
& \Phi_{i}^{(2)}-\Phi_{i+1}^{(2)}=u_{i-1}-1, 
\end{split}
\end{equation}
Here the difference $\Phi_i^{(1)}-\Phi_{i+1}^{(2)}$ is stable, and by~\eqref{e:difference}
and~\eqref{def:mys} it equals to $\mys$, so $\Phi_{i+1}^{(2)}=\Phi_i^{(1)}-\mys$ for all $i\leq l$.
Equation~\eqref{backward} follows
from $\Phi_i^{(1)}\geq\Phi_{i+1}^{(2)}$, and since $\Phi_i^{(1)}=\Phi_{i+1}^{(1)}+u_i-1$ by the
first equation of~\eqref{before-soliton}. We also obtain
\begin{equation}
\Phi_{i+1}^{(2)}\geq \Phi_i^{(2)}+u_i
\end{equation}
from the second equation of~\eqref{before-soliton}, 
since $u_i+u_{i-1}\leq 1$ for $i\leq l$, which together with $\Phi_{i+1}^{(2)}=\Phi_i^{(1)}-\mys$ 
makes~\eqref{backward}. 
\end{proof}

This result implies that in the case of one massive soliton, when $U$ satisfies (A$U$), 
a solution $\Phi$ to~\eqref{e:willox} satisfying (A$\Phi$) exists and is {\bf unique}.

\subsection{Case (C2), several solitons}

We have seen above that in the case of one soliton, the last two equations of~\eqref{e:willox} are satisfied
automatically. However, the graphs $\digr(\gamma)$ and $\digr(\delta)$ contain edges which are dangerous
to use. If $\Sat(\Phi^{(1)})$ or $\Sat(\Phi^{(2)})$ contain such edges then the last two equations 
of~\eqref{e:willox} are violated.

\begin{lemma}
\label{l:danger}
Let $\Phi^{(1)}$ and $\Phi^{(2)}$ be a solution to~\eqref{e:willox}. Then $\Sat(\Phi^{(1)})$ cannot
contain edges $(i,i+1)$ if $u_i+u_{i-1}>1$, and $\Sat(\Phi^{(2)})$ cannot contain edges $(i+1,i)$ if
$u_i+u_{i+1}>1$. 
\end{lemma} 
\begin{proof}
We use Lemma~\ref{coeffs-s2} being in the case of $(C2)$. 

If $\Sat(\Phi^{(1)})$ uses $(i,i+1)$ when $u_i+u_{i-1}>1$ then $\Phi_i^{(1)}-\Phi_{i+1}^{(1)}=-u_{i-1}$.
By~\eqref{backward} we should have $\Phi_i^{(1)}-\Phi_{i+1}^{(1)}\geq u_i-1$ and hence $u_i+u_{i-1}\leq 1$,
a contradiction.

If $\Sat(\Phi^{(2)})$ uses $(i+1,i)$ when $u_i+u_{i+1}>1$ then $\Phi_{i+1}^{(2)}-\Phi_i^{(2)}=1-u_{i+1}$.
By~\eqref{forward} we should have $\Phi_{i+1}^{(2)}-\Phi_i^{(2)}\geq u_i$ and hence $u_i+u_{i+1}\leq 1$,
a contradiction.
\end{proof}

It can be checked that the use of other edges does not lead to such contradictions, and also, using Lemma~\ref{coeffs-s1},
that there are no contradictions in the case of $(C1)$.

The following negative result is now easy to see.

\begin{theorem}
Let $U$ satisfy $(C2)$ and contain more than one soliton. Then no pair of fundamental eigenvectors associated with
a soliton can be a solution of~\eqref{e:willox}.
\end{theorem}
\begin{proof}
According to Lemma~\ref{l:danger}, to each soliton there corresponds a number of consecutive forward edges
in $\digr(\gamma)$ that cannot be used by $\Sat(\Phi^{(1)})$, located immediately after the corresponding critical
cycle in $\digr(\gamma)$. Further, there is also a number of consecutive backward edges in $\digr(\delta)$ that
cannot be used by $\Sat(\Phi^{(2)})$, located before the corresponding critical cycle in $\digr(\delta)$. 

If $\Phi^{(1)}$ and $\Phi^{(2)}$ are a pair of fundamental eigenvectors, then the switch from backward to forward edges 
can happen only once. To avoid all forbidden forward edges after the critical cycle of $\digr(\gamma)$ corresponding to the 
first (i.e., left-most) soliton, $\Sat(\Phi^{(1)})$ has to use backward edges only, which implies that $\Phi^{(1)},\Phi^{(2)}$ should be the pair associated with the first soliton. But then also $\Sat(\Phi^{(2)})$ uses all backward edges after the first critical cycle of
$\digr(\delta)$, including all forbidden backward edges corresponding to the next solitons.
\end{proof}

In the following table, we consider an example where the potential (i.e., solution of udKdV) consists of
two solitons $(\pi_1\; \pi_2)$ and $(\pi_3\;\pi_4)$, and the real numbers $\pi_1,\pi_2,\pi_3,\pi_4<1$ are 
such that $\pi_1+\pi_2>1$ and $\pi_3+\pi_4>1$.

\begin{equation*}
\begin{array}{cccccccccc}
l:	                      & 0 &     1 &     2    & 3      & 4  &     5 &     6   & 7     & 8\\      	
u_l:                          & 0 & \pi_1 & \pi_2    & 0      & 0  & \pi_3 & \pi_4   & 0     & 0 \\
\gamma_l=\min(u_l,1-u_{l-1}:  & 0 & \pi_1 & 1-\pi_1  &     0  & 0  & \pi_3 & 1-\pi_3 & 0     & 0 \\
\delta_l=\min(u_{l-1},1-u_l): & 0 &  0    & 1-\pi_2  & \pi_2  & 0  & 0     & 1-\pi_4 & \pi_4 & 0  
\end{array}
\end{equation*}

Digraphs $\digr(\gamma)$ and $\digr(\delta)$ are displayed on Figure~\ref{f:last}.

\begin{figure}
\centering
\begin{tikzpicture}[shorten >=1pt,->]
  \tikzstyle{vertex1}=[circle,fill=black!15,minimum size=17pt,inner sep=1pt]
\tikzstyle{vertex2}=[circle,fill=black!30,minimum size=17pt,inner sep=1pt]

\foreach \name/\x in {0/1, 3/5.5, 4/7, 7/11.5, 8/13} \node[vertex1] (\name) at (\x,0) {$\name$};                                                                                                                                                                                                                                                                                                                                                                                                                                                                                                                                                                                                                                                                                    \foreach \name/\x in {1/2.5, 2/4, 5/8.5, 6/10}
    \node[vertex2] (\name) at (\x,0) {$\name$};

\foreach \name/\nickname\x in {0'/0/1, 1'/1/2.5, 4'/4/7, 5'/5/8.5, 8'/8/13}
    \node[vertex1] (\name) at (\x,-3) {$\nickname$};
\foreach \name/\nickname\x in {2'/2/4, 3'/3/5.5, 6'/6/10, 7'/7/11.5}
    \node[vertex2] (\name) at (\x,-3) {$\nickname$};

\draw[xshift=1cm, yshift=-5cm] node{$0$:};
\draw[xshift=2.5cm,yshift=-5cm] node{$\pi_1$};
\draw[xshift=4cm,yshift=-5cm] node{$\pi_2$};
\draw[xshift=5.5cm,yshift=-5cm] node{$0$};
\draw[xshift=7cm,yshift=-5cm] node{$0$};
\draw[xshift=8.5cm,yshift=-5cm] node{$\pi_3$};
\draw[xshift=10cm,yshift=-5cm] node{$\pi_4$};
\draw[xshift=11.5cm,yshift=-5cm] node{$0$};
\draw[xshift=13cm,yshift=-5cm] node{$0$};




\draw (0) .. controls +(50:1cm) and +(130:1cm) .. (1);
\draw (1) .. controls +(50:1cm) and +(130:1cm) .. (2);
\begin{scope}[dashed]
\begin{scope}[red]
\draw (2) .. controls +(50:1cm) and +(130:1cm) .. (3);
\end{scope}
\draw (3) .. controls +(50:1cm) and +(130:1cm) .. (4);
\draw (4) .. controls +(50:1cm) and +(130:1cm) .. (5);
\begin{scope}[red]
\draw (6) .. controls +(50:1cm) and +(130:1cm) .. (7);
\end{scope}
\draw (7) .. controls +(50:1cm) and +(130:1cm) .. (8);
\end{scope}
\draw (5) .. controls +(50:1cm) and +(130:1cm) .. (6);

\draw[xshift=1.75cm,yshift=0.9cm] node{$-1$};
\draw[xshift=3.25cm,yshift=0.9cm] node{$\pi_1-1$};
\draw[xshift=4.75cm,yshift=0.9cm] node{$-\pi_1$};
\draw[xshift=6.25cm,yshift=0.9cm] node{$-1$};
\draw[xshift=7.75cm,yshift=0.9cm] node{$-1$};
\draw[xshift=9.25cm,yshift=0.9cm] node{$\pi_3-1$};
\draw[xshift=10.75cm,yshift=0.9cm] node{$-\pi_3$};
\draw[xshift=12.25cm,yshift=0.9cm] node{$-1$};

\draw (8) .. controls +(-130:1cm) and +(-50:1cm) .. (7);
\draw (7) .. controls +(-130:1cm) and +(-50:1cm) .. (6);
\draw (6) .. controls +(-130:1cm) and +(-50:1cm) .. (5);
\draw (5) .. controls +(-130:1cm) and +(-50:1cm) .. (4);
\draw (4) .. controls +(-130:1cm) and +(-50:1cm) .. (3);
\draw (3) .. controls +(-130:1cm) and +(-50:1cm) .. (2);
\draw (2) .. controls +(-130:1cm) and +(-50:1cm) .. (1);
\begin{scope}[dashed]
\draw (1) .. controls +(-130:1cm) and +(-50:1cm) .. (0);
\end{scope}

\draw[xshift=1.75cm,yshift=-0.9cm] node{$\pi_1$};
\draw[xshift=3.25cm,yshift=-0.9cm] node{$1-\pi_1$};
\draw[xshift=4.75cm,yshift=-0.9cm] node{$0$};
\draw[xshift=6.25cm,yshift=-0.9cm] node{$0$};
\draw[xshift=7.75cm,yshift=-0.9cm] node{$\pi_3$};
\draw[xshift=9.25cm,yshift=-0.9cm] node{$1-\pi_3$};
\draw[xshift=10.75cm,yshift=-0.9cm] node{$0$};
\draw[xshift=12.25cm,yshift=-0.9cm] node{$0$};



\draw (0') .. controls +(50:1cm) and +(130:1cm) .. (1');
\draw (1') .. controls +(50:1cm) and +(130:1cm) .. (2');
\draw (2') .. controls +(50:1cm) and +(130:1cm) .. (3');
\begin{scope}[dashed]
\draw (3') .. controls +(50:1cm) and +(130:1cm) .. (4');
\draw (4') .. controls +(50:1cm) and +(130:1cm) .. (5');
\draw (5') .. controls +(50:1cm) and +(130:1cm) .. (6');
\draw (7') .. controls +(50:1cm) and +(130:1cm) .. (8');
\end{scope}
\draw (6') .. controls +(50:1cm) and +(130:1cm) .. (7');

\draw[xshift=1.75cm,yshift=-2.1cm] node{$-1$};
\draw[xshift=3.25cm,yshift=-2.1cm] node{$-1$};
\draw[xshift=4.75cm,yshift=-2.1cm] node{$-\pi_2$};
\draw[xshift=6.25cm,yshift=-2.1cm] node{$\pi_2-1$};
\draw[xshift=7.75cm,yshift=-2.1cm] node{$-1$};
\draw[xshift=9.25cm,yshift=-2.1cm] node{$-1$};
\draw[xshift=10.75cm,yshift=-2.1cm] node{$-\pi_4$};
\draw[xshift=12.25cm,yshift=-2.1cm] node{$\pi_4-1$};

\draw (8') .. controls +(-130:1cm) and +(-50:1cm) .. (7');
\draw (7') .. controls +(-130:1cm) and +(-50:1cm) .. (6');
\begin{scope}[red]
\draw (6') .. controls +(-130:1cm) and +(-50:1cm) .. (5');
\end{scope}
\draw (5') .. controls +(-130:1cm) and +(-50:1cm) .. (4');
\draw (4') .. controls +(-130:1cm) and +(-50:1cm) .. (3');
\draw (3') .. controls +(-130:1cm) and +(-50:1cm) .. (2');
\begin{scope}[red,dashed]
\draw (2') .. controls +(-130:1cm) and +(-50:1cm) .. (1');
\end{scope}
\begin{scope}[dashed]
\draw (1') .. controls +(-130:1cm) and +(-50:1cm) .. (0');
\end{scope}

\draw[xshift=1.75cm,yshift=-3.9cm] node{$0$};
\draw[xshift=3.25cm,yshift=-3.9cm] node{$1-\pi_2$};
\draw[xshift=4.75cm,yshift=-3.9cm] node{$\pi_2$};
\draw[xshift=6.25cm,yshift=-3.9cm] node{$0$};
\draw[xshift=7.75cm,yshift=-3.9cm] node{$0$};
\draw[xshift=9.25cm,yshift=-3.9cm] node{$1-\pi_4$};
\draw[xshift=10.75cm,yshift=-3.9cm] node{$\pi_4$};
\draw[xshift=12.25cm,yshift=-3.9cm] node{$0$};

\if{
\draw (1) -- (2');
\draw (2) -- (3');
\draw (5) -- (6');
\draw (6) -- (7');
}\fi

\draw[xshift=4cm,yshift=-1.5cm] node{sol1};
\draw[xshift=10cm,yshift=-1.5cm] node{sol2};

\end{tikzpicture}
\caption{\label{f:last} The case of two solitons: ``dangerous'' edges} 
\end{figure}
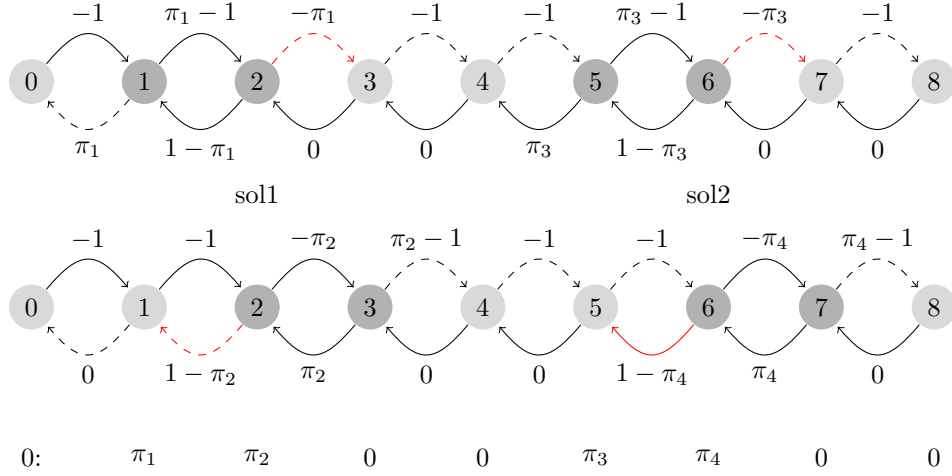

Here the bold edges are used by the pair of
fundamental eigenvectors corresponding to the first soliton $(\pi_1,\;\pi_2)$. The red edges are the ones which
lead to a contradiction with~\eqref{backward} (graph $\digr(\gamma)$, upper part of the figure) 
or~\eqref{forward} (graph $\digr(\delta)$, lower part of the figure). We see that the pair of fundamental eigenvectors
has to use one of the forbidden edges, hence it cannot satisfy both~\eqref{backward} and~\eqref{forward}.

\section{Conclusions and projects}

In this paper we attempted to build the max-plus theory of~\eqref{e:willox}. Based on the observation that the first two equations
represent max-plus spectral problems, we explained how the finite-dimensional max-plus spectral theory applies to them.
We studied pairs of fundamental eigenvectors associated with each soliton, describing the undressing transform and showing
that these pairs yield a solution of~\eqref{e:willox} in some situations.

The remaining nontrivial case is when $U^{(t)}$ has several massive solitons, where we have shown that the pairs of fundamental
eigenvectors violate the last two equations of~\eqref{e:willox}. Willox et al.~\cite{Wil+} report that a solution
can be found also in this nontrivial case. It is desirable to work out a systematic comprehensive approach to solving~\eqref{e:willox}
in this case, and in particular, to understand whether a max-plus linear combination of fundamental pairs could be a solution.
Then one could proceed with the study of undressing transform associated with any solution of~\eqref{e:willox}, and the details of application of this theory to solving the ultradiscrete KdV equation~\eqref{e:cell-aut}.

\section{Acknowledgement}

The author is grateful to Jonathan Nimmo and Ralph Willox for introducing him into the subject, and
to St\'{e}phane Gaubert for giving him an idea how the max-plus spectral theory could be applied to~\eqref{e:willox}. The author wishes to thank the anonymous referee and Anna Kazeykina for careful reading and
good questions.


\end{document}